\documentclass[12pt]{amsart}

\usepackage{graphics}
\usepackage{epstopdf}
\usepackage{graphicx}

\usepackage{latexsym}
\usepackage{amssymb}
\usepackage{euscript}

\def\mcc{M\raise.5ex\hbox{c}C}
\def\mccarthy{M\raise.5ex\hbox{c}Carthy}
\def\Hu{\H}
\def\sz{Szeg\Hu{o} }

\def\ie{{\it i.e. }}

\def\hi{H^\infty}

\def\z{\zeta}

\let\i=\infty

\def\la{\langle}
\def\ra{\rangle}
\def\={\ = \ }

\def\C{\mathbb C}

\def\D{\mathbb D}
\def\B{\mathbb B}

\def\inn{\ \in \ }

\def\dis{\displaystyle}

\def\be{\setcounter{equation}{\value{theorem}} \begin{equation}}
\def\ee{\end{equation} \addtocounter{theorem}{1}}
\def\beq{\begin{eqnarray*}}
\def\eeq{\end{eqnarray*}}

\def\bp{{\sc Proof: }}
\def\ep{{}{\hfill $\Box$} \vskip 5pt \par}

\def\bl{\begin{lemma}}
\def\el{\end{lemma}}
\def\bt{\begin{theorem}}
\def\et{\end{theorem}}
\def\bprop{\begin{prop}}
\def\eprop{\end{prop}}
\def\bd{\begin{definition}}
\def\ed{\end{definition}}
\def\br{\begin{remark}}
\def\er{\end{remark}}
\def\bexer{\begin{exercise}}
\def\eexer{\end{exercise}}
\def\bfig{\begin{figure}}
\def\efig{\end{figure}}

\newtheorem{theorem}{Theorem}[section]
\newtheorem{prop}[theorem]{Proposition}
\newtheorem{lemma}[theorem]{Lemma}
\newtheorem{cor}[theorem]{Corollary}

\newtheorem{question}[theorem]{Question}
\newtheorem{definition}[theorem]{Definition}

\theoremstyle{definition}
\newtheorem{example}[theorem]{Example}
\newtheorem{remark}[theorem]{Remark}

\renewcommand\S{\Sigma}

\def\htd{H^2_d}
\def\hts{{H^2(\Sigma)}}
\def\hths{{H^2_h(\Sigma)}}
\def\kd{k_d}
\def\bd{{\mathbb B}_d}
\newcommand{\mv}{{\mathcal M}_V}
\newcommand{\mw}{{\mathcal M}_W}
\newcommand{\cv}{\overline{V}}
\newcommand{\pv}{{\partial V}}
\newcommand{\cw}{\overline{W}}
\newcommand{\pw}{{\partial W}}
\newcommand{\cs}{\overline{\Sigma}}
\newcommand{\ps}{{\partial \Sigma}}
\newcommand\hS{{\hat \Sigma}}

\newcommand{\cA}{\mathcal{A}}

\newcommand{\cM}{\mathcal{M}}
\newcommand{\cF}{\mathcal{F}}
\newcommand{\cK}{\mathcal{K}}
\newcommand{\cT}{\mathcal{T}}
\newcommand{\Mult}{\operatorname{Mult}}
\newcommand{\spn}{\operatorname{span}}
\newcommand{\lel}{\left\langle}
\newcommand{\rir}{\right\rangle}
\newcommand{\mb}[1]{\mathbb{#1}}
\newcommand{\ol}{\overline}
\newcommand{\wot}{\textsc{wot}}
\renewcommand{\o}{\omega}
\newcommand{\MMM}{\Delta}

\begin{document}

\title[The Isomorphism Question]{On the isomorphism question for complete Pick multiplier algebras}

\author{Matt Kerr}
\thanks{Partially supported by National Science Foundation Grant
DMS 1068974}
\address
{Washington University\\
St. Louis, MO 63130\\
USA}
\email{mattkerr@math.wustl.edu}
\author{John E. M\raise.45ex\hbox{{c}}Carthy}
\thanks{Partially supported by National Science Foundation Grant
DMS 0966845}
\address{Washington University\\
St. Louis, MO 63130\\
USA}
\email{mccarthy@wustl.edu}
\author{
Orr Moshe Shalit}
\thanks{Partially supported by Israel Science Foundation Grant no. 474/12 and by the European Union Seventh Framework Programme ({\em FP7/2007-2013}) under grant agreement no. 321749}
\address{
Ben-Gurion University of the Negev\\
Beer-Sheva\\
ISRAEL}
\email{oshalit@math.bgu.ac.il}
\date{February 13, 2012}

\bibliographystyle{plain}
\begin{abstract}
Every multiplier algebra of an irreducible complete Pick kernel arises as the restriction algebra $\mv = \{f\big|_V : f \in \cM_d\}$, where $d$ is some integer or $\infty$, $\cM_d$ is the multiplier algebra of the Drury-Arveson space $H^2_d$, and $V$ is a subvariety of the unit ball. For finite $d$ it is known that, under mild assumptions, every isomorphism between two such algebras $\mv$ and $\mw$ is induced by a biholomorphism between $W$ and $V$. In this paper we consider the converse, and obtain positive results in two directions. The first deals with the case where $V$ is the proper image of a finite Riemann surface. The second deals with the case where $V$ is a  disjoint union of varieties. 
\end{abstract}

\maketitle

\section{Introduction and statement of the main results}

The general theme of this note is the following   question.

(Q) Given two specific function algebras $\mv$ and $\mw$
associated with two analytic varieties $V$ and $W$, 
if there is a biholomorphic equivalence $h: V \to W$, does composition
with $h$ induce an isomorphism from $\mw$ to $\mv$?

This type of question --- and its converse, whether isomorphism of the algebras implies biholomorphic equivalence of the varieties ---
 can be asked in many contexts, {\em e.g.} in algebraic geometry it is well known that two affine algebraic varieties are isomorphic if and only if the algebras of polynomial functions on the varieties are isomorphic
\cite[Cor. I.3.7]{har77}. Coming back to complex analysis, a theorem of L.~Bers \cite{bers48} says that two domains $V$ and $W$ in the plane are conformally equivalent if and only if the corresponding algebras of holomorphic functions $O(V)$ and $O(W)$ are isomorphic. Note that in the setting of Bers's theorem, (Q) is trivial and it is the converse that is the interesting part. In our
 setting, (Q) turns out to be harder than its converse (though the converse is by no means trivial).

Here we shall consider multiplier algebras of a certain Hilbert space associated with a variety $V$ in the unit ball $\bd$.

Let $\htd$ be the reproducing kernel Hilbert space on $\bd$, 
the unit ball of $\mb{C}^d$, with kernel
\[
\kd(z,w) \= \frac{1}{1 - \la z, w \ra} .
\]
This space was first introduced by S. Drury \cite{dru83}, and later studied
in \cite{po91, arv98, agmc_loc}.
Let $\cM_d$ be the multiplier algebra $\Mult(H^2_d)$ of $H^2_d$. 
Here $d$ is some integer or $\i$.

We shall use the word {\em variety} to mean a joint zero set in the ball of $\htd$ functions. That is, $V$ is a variety in our sense if there is some set $E \subseteq H^2_d$ such that
\[
V = \{z \in \mb{B}_d : \forall f \in E ,\, f(z) = 0\}. 
\]
By \cite[Theorem 9.27]{ampi} this is the same as being a joint zero set of functions in $\cM_d$. We shall denote by $\ol{V}$ the closure of $V$ in $\overline{\mb{B}}_d$. 

Following \cite[Section 5]{DRS12}, let us say that a variety $V$ is {\em irreducible} if for any regular point $\lambda \in V$,  and any non-empty ball $B$  centered at $\lambda$ and contained in $\mb{B}_d$, the intersection of  the zero sets of all multipliers vanishing on  $V\cap B$ is exactly $V$. 
Note that if $V$ cannot be written as the proper union of two analytic varieties {\em in the classical sense} then $V$ is irreducible according to our definition. Indeed, this follows from the fact that in this case the set of regular points of $V$ is a connected submanifold (see \cite[Section III.C]{gur}), thus any analytic function vanishing on a neighbourhood in $V$ of a regular point vanishes everywhere on $V$.

If $V$ is a variety we define 
\[
\cF_V = \ol{\spn}\{k_d(\cdot,\lambda) : \lambda \in V\}.
\]
The Hilbert space $\cF_V$ is naturally a reproducing kernel Hilbert space of 
functions on the variety $V$. We define 
\[
\mv = \{f\big|_V : f \in \cM_d\}. 
\]
By 
 \cite[Proposition 2.6]{DRS12},
$\cM_V$ is equal to $\Mult(\cF_V)$, and is completely isometrically isomorphic to the quotient of $\cM_d$ by the ideal
\[
J_V = \{f \in \cM_d : f\big|_V = 0\}. 
\]
By universality of $H^2_d$, every multiplier algebra of an irreducible complete Pick kernel is of the form $\mv$ for some $d$ and some variety $V \subseteq \mb{B}_d$ (see \cite{agmc_cnp}). 

Question (Q) in the setting just described was first considered by K.~Dav-idson, C.~Ramsey and O.~Shalit in \cite{DRS11} for the case where $V$ and $W$ are homogeneous varieties. Under additional technical assumptions it was shown that the answer to question (Q) is {\em yes}. The additional technical assumptions were later removed by M.~Hartz \cite{hartz2012}. In \cite{DRS12} question (Q) and its converse were treated in greater generality, and it was shown that, when $d<\infty$, the converse direction holds true for varieties which are finite unions of irreducible varieties plus a discrete variety. Moreover, it was shown that the answer to (Q) is positive in some very special cases, and that it is negative in some rather pathological cases. The goal of the present work is to obtain a positive answer to (Q) for a wide class of ``nice" varieties. 

In Theorem~\ref{thmd1} and Corollary~\ref{cord2}, we answer question (Q)  in the affirmative in the case that $V$ is a finite Riemann surface
with a finite number of pairs of points identified (under the assumption of sufficient regularity on the boundary).  
This generalizes work of D.~Alpay, M.~Putinar and V.~Vinnikov \cite{apv}, who 
proved the result when $V$ was the unit disk, and N.~Arcozzi, R. Rochberg and E. Sawyer \cite[Subsection 2.3.6]{ars08}, who proved the result for $V$ a planar domain. As a corollary we obtain an extension theorem: {\em every bounded holomorphic function on $V$ extends to a function in $\cM_d$} (see Corollary~\ref{cord3}).
We also prove that the converse to (Q) holds for $H^\infty$ of finite Riemann surfaces in Corollary~\ref{cord1}. 
Our work requires an understanding of Hardy spaces on finite Riemann surfaces --- we give 
background information on this in Section~\ref{secc}.

In Section \ref{secf}, we discuss varieties with a finite number of disjoint components, 
and find sufficient conditions for an isomorphism on each separate component to give
an isomorphism on the whole variety. For this we require some results on the maximal ideal spaces of $\cM_V$ and of its norm closed analogue $\cA_V$, results which are obtained in Section \ref{sece}.

\section{Isomorphism between $\mv$ and $\mw$}
\label{secb}

In this section, we shall assume that $d < \infty$.

Let us denote by $\MMM(\cM_V)$ the {\em character space} of $\cM_V$, that is, the space of all nonzero multiplicative linear functionals on $\cM_V$. 
All characters are unital and contractive, hence completely contractive (since every contractive functional of an operator space
is a complete contraction \cite[Proposition 3.8]{pau02}). Note that, since the algebras are semi-simple, every homomorphism $\alpha : \cM_V \rightarrow \cM_W$ is automatically norm continuous. Thus, if $\alpha : \cM_V \rightarrow \cM_W$ is a unital homomorphism, then there is an induced continuous map
\[
\alpha^* : \MMM(\mw) \rightarrow \MMM(\mv)
\]
defined by $\alpha^*(\rho) = \rho \circ \alpha$ for all $\rho \in \MMM(\mw)$. 

Let us denote by $z_1, \ldots, z_d$ the coordinate functions in $\mv$. Then as $(z_1, \ldots, z_d)$ is a row contraction and since all characters are completely contractive, $(\rho(z_1), \ldots, \rho(z_d))$ is in $\ol{\mb{B}_d}$ for every $\rho \in \MMM(\mv)$. This gives rise to a continuous map 
\[
\pi : \MMM(\mv) \rightarrow \ol{\mb{B}_d} . 
\]
For every $\lambda \in \mb{B}_d$ we call $\pi^{-1}(\lambda)$ {\em the fiber over $\lambda$}. In Section 3 of \cite{DRS12} it is shown that $\pi(\MMM(\mv)) \cap \mb{B}_d = V$, that the fiber over every $\lambda \in V$ is a singleton, and that $\pi$ identifies the space of \wot -continuous characters with $V$. 
If $V$ extends analytically across $\partial {\mb{B}_d}$, by Corollary~\ref{cor:pi_character} below $\pi(\MMM(\mv)) =  \ol{V}$. 
It is easy to see that $\pi(\MMM(\mv)) \supseteq  \ol{V}$ always holds. 

\begin{question}
Is $\pi(\MMM(\mv)) =  \ol{V}$ for all $V$?
\end{question}

\bt
\label{thmb1}
Let $V,W \subseteq \mb{B}_d$ be varieties, $d < \infty$. Suppose that $V$ and $W$ are each a union of finitely many varieties, each of which is either irreducible or discrete. Let $\alpha$ be an isomorphism between $\mv$ and $\mw$. 
Then $\alpha$ induces a holomorphic map $h : \mb{B}_d \to \mb{C}^d$ that restricts to a biholomorphism from $W$ onto $V$. The isomorphism $\alpha$ is implemented by composition with $h$. Moreover, if $\big[h_1, \ldots, h_d \big] = h$ are the components of $h$, then $h_i \in \cM_d$ for all $i=1, \ldots, d$. 
\et

\bp This theorem and some variants are treated in detail in Section 5 of \cite{DRS12}, but we would like to repeat some of the arguments. Consider the coordinate functions $z_i := z_i \big|_V$ in $\cM_V$.
For all $i$, $\alpha (z_i) \in \cM_W$, and is therefore the restriction to $W$ of multiplier $h_i \in \cM_d$: 
\[
\alpha(z_i) = h_i\big|_V .
\]
This gives rise to a holomorphic function
$h = \big[ h_1, \ldots, h_d \big] : \mb{B}_d \rightarrow \mb{C}^d$ with components in $\cM_{d}$. 

Fix $\lambda \in W$, and let $\rho_\lambda$ be the evaluation functional 
at $\lambda$ on $\cM_W$. 
Then $\alpha^*(\rho_\lambda)$ is a character in $\MMM(\cM_V)$. 
We now compute
\[
 [\alpha^*(\rho_\lambda)](z_i) = \rho_\lambda( \alpha(z_i)) = \alpha(z_i) (\lambda) .
\] 
So $\alpha^*(\rho_\lambda)$ is in the fiber over $(\alpha(z_1)(\lambda), \ldots, \alpha(z_d)(\lambda)) = h(\lambda)$. 

In \cite[Section 5]{DRS12} it is proved that $h(\lambda) \in V$, and therefore that  $\alpha^*(\rho_\lambda)$ is also an evaluation functional. 
Thus, for every $f \in \cM_V$ and every $\lambda \in W$, 
\begin{align*}
 \alpha(f) (\lambda) &= \rho_\lambda (\alpha(f)) 
 = \alpha^*( \rho_\lambda) (f) \\
 &= \rho_{h(\lambda)} (f)   =  (f \circ h)(\lambda) .
\end{align*}
Therefore $\alpha(f) = f \circ h$. 

Finally, since the same reasoning holds for $\alpha^{-1}$, $h$ is a biholomorphism. 
\ep

In general, the function $h$ appearing in the above theorem extends to a homeomorphism between the character spaces, but not to a homeomorphism between $\ol{V}$ and $\ol{W}$. Indeed, suppose that $V$ and $W$ are both interpolating sequences. Then $\cM_V$ and $\cM_W$ are isomorphic (both are isomorphic to $\ell^\infty$), and consequently $V$ and $W$ are biholomorphic, but $\ol{V}$ and $\ol{W}$ need not be homeomorphic. 

Let $\cA_V$ be the norm closure of the polynomials in $\cM_V$.

\begin{question}
When is $\cA_V = \mv \cap C(\ol{V})$?
\end{question}
We will see in Section \ref{sec:charAV} that always $\cA_V \subseteq \mv \cap C(\ol{V})$.

\bt
\label{thmb1_more}
Let $V,W \subseteq \mb{B}_d$ be varieties, $d < \infty$. Let $\alpha$ be an isomorphism between $\mv$ and $\mw$, and let $h$ be as in Theorem \ref{thmb1}. Then
 $\alpha$ maps $\cA_V$ into $\cA_W$,
(respectively  $\mv \cap C(\ol{V})$ into $\mw \cap C(\ol{W})$),
 if and only if $h_i\big|_W \in \cA_W$ (respectively, $\mw \cap C(\ol{W})$)  for all $i=1, \ldots, d$. 
Moreover, $\alpha$ is an isomorphism from  $\mv \cap C(\ol{V})$ onto $\mw \cap C(\ol{W})$ if and only if $h\big|_W$ extends to be a homeomorphism from $\cw$ to $\cv$. 
\et
\bp We use the notation of the proof of the above theorem. If $\alpha$ takes $\cA_V$ into $\cA_W$ then for all $i$, $h_i\big|_W = \alpha(z_i) \in \cA_W$ (similarly for $\mv \cap C(\ol{V})$ and $\mw \cap C(\ol{W})$). 

 Conversely, if $h_i\big|_W = \alpha(z_i) \in \cA_W$ for all $i$, then, since $\alpha$ is continuous, we have that $\alpha(\cA_V) \subseteq \cA_W$. 
If  $h_i\big|_W \in \mw \cap C(\ol{W})$, then $f \circ h$ is in $\mv \cap C(\ol{V})$ for every $f$ that is in  $ \mv \cap C(\ol{V})$, so $\alpha(\mv \cap C(\ol{V})) \subseteq \mw \cap C(\ol{W})$.

If $\alpha : \alpha(\mv \cap C(\ol{V})) \to \mw \cap C(\ol{W})$ is an isomorphism, then $h$ must be continuous, and
one-to-one, on $\ol{W}$, and so must $h^{-1}$, therefore $h$ is a homeomorphism.
\ep

\section{Background on finite Riemann surfaces}
\label{secc}

A connected {\em finite Riemann surface} $\S$ is a connected open
proper subset of some  compact Riemann surface such that the boundary $\partial \S$ is
also the boundary of the closure and is the union of finitely many disjoint simple closed
analytic curves. A general finite Riemann surface is a finite disjoint union of connected ones.
 Figure \ref{picc1} shows a picture of a connected finite Riemann surface with three
boundary circles.
We shall identify $\Sigma$ with its image in the double, $\hS$, the compact Riemann surface
obtained by attaching another copy of $\S$ along $\partial S$.

\begin{figure}[h]
\includegraphics[scale=0.45]{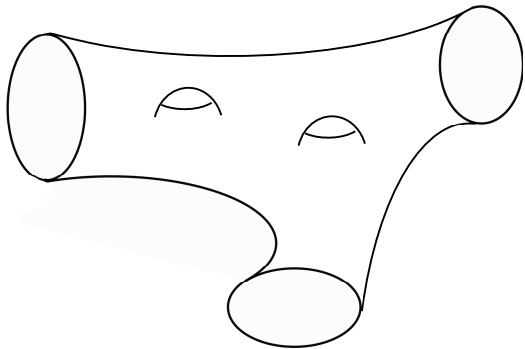} 
\centering
\caption{Finite Riemann Surface with two handles and three boundary components} 
\label{picc1} 
\end{figure}

Let $O(\Sigma)$ denote the algebra of all holomorphic functions on $\Sigma$, and $A(\Sigma)$ the 
holomorphic functions that extend continuously to $\partial \S$. Since $\S$ is not compact, there
is a plentiful supply of holomorphic functions; indeed, by a result of J. Wermer \cite{wer64}, the 
real parts of functions in $A(\S)$ form a finite codimensional subspace of the space of all
continuous real-valued functions on $\partial \S$. This means that in many important ways
function theory on $\S$ is like function theory on the unit disk, subject to a finite number of extra conditions. See for example the paper \cite{ahesar1}.

Let us fix some base-point $a$ in $\S$, and let $\o$ be the harmonic measure with respect to $a$,
\ie the measure on $\ps$ with the property that
\[
\int_\ps u(\zeta) d \omega(\zeta ) \= u(a) 
\]
for every function $u$ that is harmonic on $\S$ and continuous on $\cs$.

%

We shall define $H^2(\S)$ to be the closure in $L^2(\o)$ of $A(\S)$. Its multiplier algebra is
$H^\i(\S)$, the bounded analytic functions on $\S$.
Note that the norm in $H^2(\S)$ depends on the choice of base-point $a$, but the norm
in $H^\i(\S)$ does not --- it is the supremum of the modulus on $\S$ \cite{ahesar1}.

J.D. Fay \cite[Prop 6.15]{fay} gave an explicit formula for the \sz kernel for $H^2(\S)$, \ie
for the function $j(\z,\eta)$ that is conjugate-holomorphic in $\eta$,
as a function of $\z$ is holomorphic and  lies in $H^2(\S)$ for each fixed $\eta$, 
and satisfies
\be
\label{eq1a}
\int_\ps f(z) \overline{j(\z,\eta)} d \omega(\z)
\= f(\eta) \qquad \forall f \in H^2(\S) .
\ee

We need to know how $j$ behaves qualitatively near the boundary.
Fix some point $\eta$ in $\ps$, and let us introduce local coordinates.
Let $U$ be a small neighborhood of $\eta$ in $\hS$, and $V$ a neighborhood of an arc
of the unit circle in the complex plane, and $\phi : V \to U$ a biholomorphic map that
sends $V \cap \D$ to $U \cap \S$, and $V \cap \D^c$ to $U \cap \S^c$.
Let us write $\z = \phi(z)$ and $\eta = \phi(w)$. Let  $J(z,w) = j(\phi(z),\phi(w))$.

\begin{prop}
\label{corkey}
With notation as above, let $w_0 \in V \cap \partial \D$.
Then there is a neighborhood $V_0$ of $w_0$ that is contained in $V$ and so that
on $V_0 \times (V_0 \cap \partial \D) $ we have
\be
\label{eqd4}
J(z,w) \= \frac{G(z,w)}{w-z} 
\ee
where  $G$ is a holomorphic function that is  non-vanishing  on $V \times (V_0 \cap \partial \D)$.
\end{prop}

Proposition~\ref{corkey} is a corollary of Fay's theorem, which gives an explicit formula for $j$ in terms of the theta function $\theta$ and the prime form $E$ of $\S$:

\bt [Fay]
\label{thmfay}
The reproducing kernel $j(\z,\eta)$ for $H^2(\S)$ is given by:

\be
\label{formfay2}
j(\z,\eta) \=
\frac{\theta(\eta^* - \z - e) \theta(a^* -a - e) E(\z,a^*) E(\eta^*,a)}{
\theta(a^*- \z - e)\theta(\eta^* -a - e) E(\eta^*,\z) E(a,a^*)}.
\ee
Moreover, if $U = U^*$ is a small symmetric neighborhood of some point in $\ps$,
then on $U \times U$, the function $k(\z,\eta) := j(\z,\eta^*)$ is meromorphic,
non-vanishing, and is holomorphic except for simple poles where $\eta = \z$.
\et

Here $e$ is a certain  even half-period, 
$E$ is the prime form, 
 which  in local coordinates has the expansion
\[
E(z,w) \=  \overline{ E(z^*, w^*)} \= \frac{ (z-w) + O(z-w)^2}{\sqrt{dz dw}} ,
\]
 $\theta$ is the theta function for $\hS$, and we use 
 $ \zeta^*$ to mean the involution on $\hS$
that swaps $\Sigma$
and $\hS \setminus \overline{\Sigma}$.
For details see Fay's monograph \cite{fay}.

A Toeplitz operator on $\hts$ is the operator of multiplication by a continuous function $f$
on $\ps$ followed by projection from $L^2(\sigma)$ onto $\hts$; we denote this operator by
$T_f$. The $C^*$-algebra
generated by all the Toeplitz operators is called the Toeplitz algebra and is denoted $\cT$. Let us denote the algebra of compact operators on $\hts$ by $\cK$. We  need to know that $\cT$ is an extension of $\cT \cap \cK$ by 
the continuous functions on $\ps$. This result is well-known, but we do not know of a convenient
reference, so we include a proof for completeness.

\bprop
\label{propt}
The quotient $\cT / (\cK \cap \cT)$ is isometrically isomorphic to $C(\ps)$ via the map $f \leftrightarrow T_f + \cK \cap \cT$.
\eprop
\bp
If $\Sigma$ is the disjoint union of $\Sigma_1, \ldots, \Sigma_k$, then every Toeplitz operator on $\hts$ decomposes into a direct sum of Toeplitz operators on $H^2(\Sigma_i)$, $i=1, \ldots, k$. Thus, we may assume that $\Sigma$ is connected.  

 For an inner function $B$ in $A(\S)$  (inner means that the modulus is $1$ on $\ps$), the codimension of
$B \hts$ in $\hts$  is equal to the number of zeroes of $B$ in $\S$, and hence finite.
It follows that for such a $B$, the operator $T_B$ is essentially normal,
because the kernel of $T_B T_{B^*} - T_{B^*} T_B$ is 
$B \hts$.

The Toeplitz algebra $\cT$ is irreducible. Indeed, any operator in
the commutant of $A(\Sigma)$ must be (multiplication by) a function
in  $H^\infty(\Sigma)$,
so the commutant of $\cT$ is contained in the intersection of
$H^\infty(\Sigma)$ with its complex conjugate, and since $\S$ is connected,
this reduces to the scalars.
Therefore $\cK \subseteq \cT$. 

A result of E. L. Stout \cite{sto66} is that every connected finite Riemann surface
can be embedded in $\C^3$ with a triple $B_1, B_2, B_3$ of inner functions in $A(\S)$. We may therefore assume that $\Sigma \subseteq \mb{D}^3$ and that $B_i$ is equal to the coordinate function $z_i$. 
Then $B_1, B_2, B_3$ generate $A(\Sigma)$, and these functions and their conjugates generate $C(\Sigma)$. Since $T_g T_f = T_{gf}$ for every $f \in A(\Sigma)$ and $g \in C(\ps)$, it follows that $T_{B_1}, T_{B_2}, T_{B_3}$ generate the C*-algebra $\cT$. It follows from $T_B T_{B^*} - T_{B^*}T_{B} \in \cK$ that $\cT / \cK$ is commutative. The spectrum of this commutative C*-algebra can be naturally identified with $\ps$, thus the result follows.
\ep

\section{Biholomorphic maps induce isomorphisms}

In \cite{apv}, Alpay, Putinar and Vinnikov showed
 that if $h: \D \to W$ is a biholomorphic unramified $C^2$ map that is transversal at the boundary,
then there is an isomorphism of multiplier algebras from ${\mathcal M}_\D = \hi(\D)$ to $\mw$.
Arcozzi, Rochberg and Sawyer extended this result by allowing $\D$ to be replaced by a finitely
connected planar domain \cite{ars08}. 
We extend one step further, to the case of a finite Riemann surface.
Note that we do not require $d$ to be finite.

In \cite{amhyp}, the following definition was introduced and studied:
\begin{definition}
\label{defa2}
A holomap is a proper holomorphic map $h$ from a 
Riemann surface $\Sigma$ 
into a bounded open set $U \subseteq\C^n$ such that
there is a finite subset $E$ of $\Sigma$ with the property that
 $h$ is non-singular
and injective on $\Sigma \setminus E$.
\end{definition}
If $h$ is a holomap from $\Sigma$ to $W$ and $X$ is a space of holomorphic functions on $\Sigma$, then by $X_h$ we mean the functions in $X$ that are
of the form $F \circ h$ for some holomorphic function $F$ on $W$
(thus we get $H^p_h(\Sigma), A_h(\Sigma)$, {\em etc}).

\bt
\label{thmd1}
Let $\S$ be a finite Riemann surface, and $W$ 
a variety in $\B_d$.
Let $h: \S \to \B_d$ be a holomap that maps $\S$ onto $W$,  is $C^2$ up to $\partial \S$, and
is one-to-one on $\partial \S$.
Assume that $h$ is transversal, \ie
\be
\label{eqd1}
\la D h(\zeta), h(\zeta) \ra \ \neq \ 0,\qquad \forall \zeta \inn \partial \S.
\ee
Then the map 
\[
\alpha: f \mapsto f \circ h
\]
is an isomorphism from $\cF_W$ onto $H^2_h(\S)$.
Consequently, the mapping $\phi \mapsto \phi \circ h$ implements an isomorphism of $\mw$ onto $H^\infty_h(\S)$.
\et

\begin{remark}
 The calculation of $D h$, the derivative of $h$,
depends on the choice of local coordinate, and another
choice will differ by a non-zero multiplicative factor (the derivative of the change of variable map).
This will not, however, affect whether or not 
$\la D h(\zeta), h(\zeta) \ra$ is non-zero.
\end{remark}
\begin{remark}
One can think of $W$ as the space $\Sigma$ with a finite number of pairs of points identified,
or limiting cases of this where one gets cusps. See \cite{amhyp} for a detailed description.
\end{remark}

\bp
The main idea of the proof goes back to \cite{apv}. To show that $\alpha$ given by the formula $f \mapsto f \circ h$ is a well defined bounded and invertible map from $\cF_W$ onto $\hths$, we will formally compute $\alpha^*$ and $R := \alpha \alpha^*$, and show that $R$ is an injective Fredholm operator. Being positive and Fredholm, injectivity implies invertibility, and the first claim in the theorem follows. A straightforward computation then shows that the asserted isomorphism between $\cM_W$ and $H^\infty_h(\S)$ is  the similarity induced by $\alpha$. 

For $f$ in $\hts$ (not necessarily in $\hths$) let us calculate $\alpha^*(f)$. This is the function on $W$ whose value at a point $w$ is given by
\[
\la \alpha^*f, k_d( \cdot, w) \ra_{H^2_d} \= 
\la f (\z), k_d(h(\z), w) \ra_{\hts}
\]
Let $j(\zeta,\eta) = j_\eta(\z)$ be the reproducing kernel for $H^2(\S)$, defined by
(\ref{eq1a}). We have that
\beq
\alpha^* j_\eta(w) &\=&
\la j_\eta (\z) , k_d(h(\z), w) \ra \\
&=& 
\frac{1}{1- \la w, h(\eta) \ra} \\
&=& 
k_d(w,h(\eta)) .
\eeq
Let $R = \alpha \alpha^*$. We have shown
\be
\label{eqd2}
Rf (\z) \= \int_{\ps} f(\eta) \frac{1}{1- \la h(\z), h(\eta) \ra} d\omega (\eta) .
\ee
Let us write 
\[
L(\z,\eta) \= \frac{1}{j(\zeta,\eta)} \frac{1}{1-\la h(\z), h(\eta)\ra} .
\]
For $\eta \in \ps$, let $L(\eta, \eta) = \lim_{\z \to \eta} L(\z,\eta)$.
Break $R$ up as 
\beq
R_1 f(\z) &\=&  \int_{\ps} f(\eta) j(\z,\eta) L(\eta,\eta) d\omega (\eta) \\
R_2  f(\z) &\=&  \int_{\ps} f(\eta) j(\z,\eta)\left[ L(\zeta, \eta) -  L(\eta,\eta)\right] d\omega (\eta).
\eeq

Fix some point $\eta$ in $\ps$, and let us introduce local coordinates.
Let $U$ be a small neighborhood of $\eta$ in $\hS$, and $V$ a neighborhood of an arc
of the unit circle in the complex plane, and $\phi : V \to U$ a biholomorphic map that
sends $V \cap \D$ to $U \cap \S$, and $V \cap \D^c$ to $U \cap \S^c$.
Let us write $\z = \phi(z)$ and $\eta = \phi(w)$. Let $H(z) = h(\phi(z))$
and $J(z,w) = j(\phi(z),\phi(w))$.

By Proposition~\ref{corkey}, since $h$ is $C^1$, we have
\begin{align*}
L(\eta,\eta) &= \lim_{z \to w} \frac{1}{J(z,w)}\frac{1}{1 - \lel H(z),H(w) \rir} \\
&= \lim_{z \to w} \frac{w-z}{G(w,w) }\frac{1}{\lel H(w),H(w) \rir - \lel H(z),H(w) \rir} \\
&=\lim_{z \to w}\frac{w-z}{G(w,w)\,  \lel H(w) - H(z), H(w)\rir} \\
&= \frac{1}{G(w,w)\, \la D H(w), H(w) \ra} 
\end{align*}
Therefore $R_1$ is a Toeplitz operator with a continuous non-vanishing symbol, and so 
is Fredholm by Proposition~\ref{propt}.

We claim that $R_2$ is Hilbert-Schmidt, hence compact, so $R$ is a positive Fredholm operator.
Indeed, $R_2$ is an integral operator, and will be Hilbert-Schmidt provided the function
\[
M(\zeta,\eta) :=  j(\z,\eta)\left[ L(\zeta, \eta) -  L(\eta,\eta)\right] 
\] 
is continuous on $\ps \times \ps$.
The function $M$ is clearly continuous when $\zeta$ stays away from $\eta$, so we just 
need to check the limit exists as $(\z, \eta) \in \ps \times \ps$ tends to the diagonal of $\ps \times \ps$. 

Let us use local coordinates as before.
Then we wish to calculate, for points $z,w$ and $u$ in the unit circle, 
\begin{align*}
&\lim_{(z,w) \to (u,u)} \left[
\frac{1}{1 - \la  H(z), H(w) \ra} 
-\frac{J(z,w)}{G(w,w) \, \la D H(w), H(w) \ra}\right] \\
&=
\lim_{(z,w) \to (u,u)} \left[
\frac{1}{\la H(w) - H(z), H(w) \ra} 
-\frac{G(z,w)/G(w,w)}{ (w-z) \la D H(w), H(w) \ra}\right]
\end{align*}
Using the assumption that $h$ (and therefore $H$) is $C^2$, we have that 
\[
H(z) = H(w) + (z-w) D H(w)  + (z-w)^2 \phi(z,w) 
\]
where $\phi$ is continuous. Moreover, we have
\[
G(z,w)/G(w,w) \= 1 +  (z-w)\psi(w) + O(z-w)^2 ,
\]
where 
\[
\psi(w)  \= \frac{1}{G(w,w)} \frac{\partial}{\partial z} G(z,w) |_{(w,w)}
\]
is continuous.

So the above limit becomes
\begin{align*}
\lim_{(z,w) \to (u,u)} & \frac{1}{w-z} \ \left[
\frac{1}{\la D H(w) - (w-z) \phi(z,w), H(w) \ra} \right.\\
&\left. -\ 
\frac{1 +(w-z) \psi(w) + O(w-z)^2}{ \la D H(w), H (w) \ra} \right]\\
& =
\lim_{(z,w) \to (u,u)}   \ \left[
\frac{\la \psi(w) D H(w) + \phi(z,w) , H(w) \ra}
{\la D H(u), H(u) \ra^2} \right] \\
&= \ \left[
\frac{\la \psi(u) D H(u) + \phi(u,u) , H(u) \ra}
{\la D H(u), H(u) \ra^2} \right] .
\end{align*}

So $M$, {\em a priori} defined only for $\z\neq \eta$, agrees
with a  continuous function on all of $\ps \times \ps$, as required.

Finally, suppose $f $ is in the kernel of $R$. By expanding $\dis \frac{1}{1-\la h(\z), h(\eta) \ra}$ in powers of $h(\zeta)$ and using (\ref{eqd2}), we see that
for all $\zeta \in \S$, 
\[
\sum_\gamma \frac{|\gamma|!}{\gamma!} h(\zeta)^\gamma \lel f, h^\gamma \rir = [\alpha \alpha^* f] (\zeta) = 0 .
\]
Taking the inner product of the function $\alpha \alpha^* f$ with $f$ we find
\[
\lel \alpha \alpha^* f , f \rir = \sum_\gamma \frac{|\gamma|!}{\gamma!} \lel f, h^\gamma \rir \lel h^\gamma , f \rir = 0 , 
\]
therefore $\lel f, h^\gamma \rir = 0$ for all $\gamma$. As polynomials in $h$ span $H_h^2(\S)$, this means $f \perp H_h^2(\S)$. It follows that $R$ is invertible, thus $\alpha$ is an isomorphism from $\cF_W$ onto $H^2_h(\S)$, as required. 

\ep

\begin{example}
Let $(a,b)$ be a vector in $C^2$ of length one, with neither coordinate zero,
and consider the map $h : \D \to \B_2$ given by
$h(z) = (a z^2, b z^3)$.
This is a holomap that is transversal. The algebra $H^\i_h$ consists of all bounded holomorphic
functions  on the disk whose derivatives vanish at the origin.
\end{example}

\begin{cor}
Let $W$ be as in Theorem \ref{thmd1}. Then $\cA_W = \mw \cap C(\ol{W})$. 
\end{cor}

\bp
On the one hand, $A_h(\Sigma) = H^\infty_h(\Sigma) \cap C(\ol{\Sigma})$, and this algebra is mapped by composition with $h$ isomorphically onto $\mw \cap C(\ol{W})$ (as in Theorem \ref{thmb1_more}). On the other hand, $A_h(\Sigma)$ is equal to the norm closure of $\{\textrm{polynomials} \} \circ h$ in $H^\infty_h$, and this algebra is mapped by composition with $h$ isomorphically onto $\cA_W$. The result follows. 
\ep

\begin{cor}
\label{cord1}
Let $\S_1$ and $\S_2$ be finite Riemann surfaces. 
Then $H^\i(\S_1)$ and  $H^\i(\S_2)$ are isomorphic if and only if $\S_1$ and $\S_2$ are
biholomorphic.
\end{cor}
\bp
Using Stout's result (mentioned in the proof of Proposition \ref{propt}), and replacing each inner function $B_i$ with $\frac{1}{\sqrt{3}} B_i$,
one can embed any finite Riemann surface $\S$ in $\B_3$ with a map $h$.
By the Schwarz reflection principle, each $B_i$ continues analytically across $\partial \S$,
and
 maps some neighborhood $U$ of $\overline{\S}$ in $\hS$ into $\C$
 in such a way that
$B_i(U \setminus \overline{\S} ) \subset \C \setminus \overline{\D}$.
Therefore $B_i$ cannot have a vanishing derivative anywhere on $\ps$.

Let $\zeta \in \ps$, and let $z$ be a local coordinate around $\zeta$ as in the proof of Theorem \ref{thmd1} so that $z(\zeta) = 1$. Since $|B_i(z)|^2$ is increasing as one crosses $\partial \S$, we have $\frac{\partial}{\partial r} |B_i(z)|^2 \big|_{\zeta} > 0$. But
\[
\frac{\partial}{\partial r} |B_i(z)|^2 = \frac{z}{r}B_i'\ol{B_i} + \frac{\ol{z}}{r} B_i \ol{B_i'} = 2 \Re \frac{z}{r}B_i'\ol{B_i}.
\]
It follows that $\Re B_i'\ol{B_i} > 0$ for all $i$, therefore 
\[
\la D h(\zeta), h(\zeta) \ra \= \sum_{i=1}^3 B_i^\prime(\z) \overline{B_i(\z)} \  \neq \ 0 ,
\] 
and $h$ is transversal.

So we can assume that each of $\S_1$ and $\S_2$ are embedded transversally in $\B_3$
by maps that are $C^\infty$ up to the boundary. 
By Theorem~\ref{thmd1}, the multiplier algebras $\cM_{\S_1}$
and $\cM_{\S_2}$ are isomorphic to  $H^\i(\S_1)$ and  $H^\i(\S_2)$ , respectively.
Therefore by  
Theorem~\ref{thmb1},
if $H^\i(\S_1)$ and  $H^\i(\S_2)$ are isomorphic, then
 $\S_1$ and $\S_2$ are
biholomorphic.

The converse is immediate. \ep

We can now answer question (Q) in the affirmative for a wide class of one dimensional varieties when the biholomorphism is sufficiently regular.

\begin{definition}
Let $g : W \rightarrow V$ be a holomorphic map that is $C^1$ up to $\pw$ and that maps $\pw$ onto $\pv$. We say that $g$ is {\em transversal on $\pw$} if 
for every $w \in \pw$ and every analytic curve $\alpha: \mb{D} \rightarrow W$  which extends to a $C^1$ homeomorphism from $\mb{D} \cup \{1\}$ onto $\alpha(\mb{D}) \cup \{w\}$, 
\[
\lel D (g\circ \alpha)(1), g \circ \alpha (1) \rir \neq 0 .
\]
\end{definition}

\begin{cor}
\label{cord2}
Let $\Sigma,W$ and $h$ be as in Theorem \ref{thmd1}. Let $V$ be a
variety in $\mb{B}_{d}$, and let $g : W \rightarrow V$ be biholomorphism that extends to be $C^2$, one-to-one and transversal on $\partial W$.
Then the mapping $\phi \mapsto \phi \circ g$ implements an isomorphism of $\mv$ onto $\mw$. 
\end{cor}
\bp
By Theorem \ref{thmd1}, $\mw$ is isomorphic to $H^\infty_h(\Sigma)$ via the composition map. Since $g$ is transversal on $\pw$, it follows that $g \circ h$ is transversal, so Theorem \ref{thmd1} implies that $\mv$ is isomorphic to $H^{\infty}_{g\circ h}(\Sigma)$ via the composition map. But $H^{\infty}_{g\circ h}(\Sigma) = H^{\infty}_{h}(\Sigma)$, therefore the mapping $\phi \mapsto \phi \circ (g \circ h) \circ h^{-1} = \phi \circ g$ is an isomorphism from $\mv$ to $\mw$. 
\ep

To conclude this section, we want to emphasize the following application to the extension problem of bounded holomorphic functions, which is a corollary of Theorem \ref{thmd1} (this type of result was also observed, in lesser generality, in \cite{apv,ars08}). 

\begin{cor}
\label{cord3}
Let $W$ be as in Theorem \ref{thmd1}. Then $\mw = H^\infty(W)$, and every bounded analytic function on $W$ extends to a multiplier in $\cM_d$, and in particular to a bounded holomorphic function on the ball. Moreover, the extension operator is bounded. 
\end{cor}

\section{The character space of $\cA_V$} \label{sec:charAV}
\label{sece}

In this and the next section, we assume that $d < \infty$.
Denote by $\cA_d$ the norm closure of the polynomials in $\cM_d$.

\begin{lemma}\label{lem:O}
For every $\epsilon > 0$ 
\[
O((1+\epsilon)\mb{B}_d) \subseteq \cA_d. 
\]
\end{lemma}
\begin{proof}
Since $(M_{z_1}, \ldots, M_{z_d})$ is a row contraction, this follows from the fact that the power series for every $f \in O((1+\epsilon)\mb{B}_d)$ 
converges absolutely at every point in $(1+\epsilon)\mb{B}_d$. 
\end{proof}

Recall the notation $\cA_V$ for the closure of the polynomials in $\cM_V$. 
\begin{lemma}
\label{lemd4a}
$\cA_V \subseteq \mv \cap C(\ol{V})$. 
\end{lemma}
\bp
Every $f \in \cA_V$ is the norm limit of a sequence $\{p_n\}$ of polynomials. Since the operator norm dominates the sup norm, the sequence $\{p_n\}$ converges uniformly on $V$. It follows that the sequence $\{p_n\}$ converges uniformly on $\overline{V}$. Hence $f$ can be extended continuously to all of $\overline{V}$.
\ep

We write $\MMM (\cA_V)$ for the space of multiplicative linear functionals of the algebra $\cA_V$. Since functions in $\cA_V$ are continuous on $\ol{V}$, it is clear that every $\lambda \in \ol{V}$ gives rise to an evaluation functional $\rho_\lambda$ given by 
\[
\rho_\lambda(f) = f(\lambda) \,\, , \,\, f \in \cA_V.
\]

\begin{theorem}\label{thm:character}
Assume that $V$ is a variety in $\mb{B}_d$ such that there exists a non-trivial ideal $I_V \subseteq \cA_d$ for which 
\[
\ol{V} = \{z \in \ol{\mb{B}_d} : \forall f \in I_V ,\,  f(z) = 0\}. 
\]
Then $\MMM(\cA_V)$ can be identified homeomorphically with $\overline{V}$ in the following way: to each $\rho \in \MMM(\cA_V)$ there is some $\lambda \in \overline{V}$ such that $\rho = \rho_\lambda$, and every $\lambda \in \overline{V}$ gives rise to a an evaluation functional. 
\end{theorem}

\begin{proof}
Let $\lambda \in \overline{V}$. It is clear that $\lambda$ gives rise to an evaluation character $\rho_\lambda$ of $\mv \cap C(\ol{V})$. By  Lemma~\ref{lemd4a}, $\rho_\lambda$ restricts to a character on $\cA_V$.  

Conversely, let $\rho \in \MMM(\cA_V)$. Let $\lambda = (\lambda_1, \ldots, \lambda_d)$, where for all $i$ we put $\lambda_i = \rho(z_i)$. As $\rho$ is completely contractive, $\lambda \in \overline{\mb{B}}_d$. It is clear that the restriction of $\rho$ to $\mb{C}[z]$ must be equal to the evaluation functional $\rho_\lambda$ for this $\lambda$. We will show that $\rho$ must agree with $\rho_\lambda$ on $\cA_V$, and that $\lambda \in \ol{V}$. 

Since $(z_1\big|_V, \ldots, z_d\big|_V)$ is a row contraction that generates $\cA_V$, and since $\cA_d$ is the universal operator algebra generated by a row contraction \cite{arv98}, it follows that the restriction map is a homomorphism $q : \cA_d \rightarrow \cA_V$. Therefore $\rho$ lifts to a character $\hat{\rho} = \rho \circ q$ of $\cA_d$, which, as is well known, must be evaluation at the point $\lambda$. 

To show that $\lambda$ must be in $\overline{V}$, assume for contradiction that $\lambda \notin \ol{V}$. Then there exists some $f \in I_V$ such that $f \big|_{\ol{V}} \equiv 0$ while $f(\lambda) \neq 0$. But then we have 
\[
0 = \rho(f\big|_V) = \hat{\rho}(f) = f(\lambda) \neq 0.
\]
It follows that $\lambda$ has to be in $\ol{V}$ and that $\rho = \rho_\lambda$. 

Finally, the mapping $\overline{V} \ni \lambda \mapsto \rho_\lambda \in \MMM(\cA_V)$ is bijective and clearly weak-$*$ continuous, hence a homeomorphism. 
\end{proof}

\begin{cor}\label{cor:pi_character}
Assume that $V$ is a variety in $\mb{B}_d$ such that there exists a non-trivial ideal $I_V \subseteq \cA_d$ for which 
\[
\ol{V} = \{z \in \ol{\mb{B}_d} : \forall f \in I_V ,\, f(z) = 0\}. 
\] 
If $\pi : \MMM(\mv) \rightarrow \overline{\mb{B}_d}$ is the natural projection given by $\pi(\rho) = (\rho(z_1), \ldots, \rho(z_d))$, then $\pi(\MMM(\mv)) =  \ol{V}$. 
\end{cor}

\section{Unions of strongly disjoint varieties} 
\label{secf}

In this section we will consider varieties in $\mb{B}_d$ which decompose in a very controllable way, to obtain some positive results towards question (Q). 

\begin{definition}
Let $V_1, \ldots, V_k$ be varieties in $\mb{B}_d$. We will say that $V_1, \ldots, V_k$ are {\em strongly disjoint} if for all $i=1, \ldots, k$, there exists a function $\phi_i \in \cM_d$ such that $\phi_i \equiv 1$ on $V_i$ and $\phi_i \equiv 0$ on $V_j$ for $j \neq i$. 
\end{definition}

An easy example of strongly disjoint varieties is given in the following lemma. 

\begin{lemma}\label{lem:partition1}
Suppose that $V = V_1 \cup \ldots \cup V_k$ is such that for every $i$ there is some $\epsilon > 0$ and a variety (in the usual sense) $\tilde{V}_i$ in $(1+\epsilon)\mb{B}_d$, such that $V_i = \tilde{V}_i \cap \mb{B}_d$. Assume further that $\tilde{V}_i \cap \tilde{V}_j = \emptyset$ for all $i \neq j$. Then $V_1, \ldots, V_k$ are strongly disjoint. 
\end{lemma}
\begin{proof}
Define $\phi_i$ by $\phi_i \equiv 1$ on $\tilde{V}_i$ and $\phi_i \equiv 0$ on $\tilde{V}_j$ for $j \neq i$. Then $\phi_i$ is holomorphic on $\tilde{V}$. By   Cartan's extension Theorem 
\cite[Theorem VIII.A.18]{gur} $\phi_i$ extends to a holomorphic function on $(1+\epsilon)\mb{B}_d$. By Lemma \ref{lem:O}, the restriction of $\phi_i$ to the unit ball is in $\cA_d \subset \cM_d$.
\end{proof}

Here is a stronger result.

\begin{lemma}\label{lem:partition2}
Let $V = V_1 \cup \ldots \cup V_k$, where $V_1, \ldots, V_k$ are varieties in $\mb{B}_d$ such that $\overline{V}_i \cap \overline{V}_j = \emptyset$ for all $i \neq j$. Assume further that there exist  non-trivial, finitely generated ideals $I_1, \ldots, I_k$ in $\cA_d$ for which 
\[
\ol{V_i} = \{z \in \ol{\mb{B}_d} : \forall f \in I_i , \, f(z) = 0\}. 
\] 
Then $V_1, \ldots, V_k$ are strongly disjoint. 
\end{lemma}
\bp
If $k=1$ there is nothing to prove, so let us consider $k=2$ first. For each $V_i$, there are finitely many functions $\psi_{i,n}$ 
such that $\ol{V}_i$ is the common zero set. 
If  $\ol{V}_1 \cap \ol{V}_2$ is empty, then the functions $\psi_{i,n}$ can't all lie in a maximal ideal (as this would have to be in a fiber over some point in  $\ol{V}_1 \cap \ol{V}_2$, by Theorem \ref{thm:character}), so the algebra they generate is all of $\cA_d$. Therefore one can find $g_{i,n} \in \cA_d$ such that
\[
1 = \sum_{n} g_{1,n} \psi_{1,n} + \sum_{n} g_{2,n} \psi_{2,n} .
\]
Define $\phi_1 =  \sum_{n} g_{2,n} \psi_{2,n}$ and $\phi_2 =  \sum_{n} g_{1,n} \psi_{1,n}$. Since $\psi_{i,n}$ vanishes on $V_{i}$, the functions $\phi_1, \phi_2$ satisfy the definition of strong disjointness. 

Suppose now that $k>2$. By the case $k=2$ which we just proved, for each $i \neq j$, there exist functions
$\psi_{ij}$ that are $1$ on $V_i$ and $0$ on $V_j$.
Then the functions $\phi_i = \Pi_{j \neq i} \psi_{ij}$ satisfy the definition of strong disjointness.
\ep

\begin{theorem}\label{thm:restrictions}
Let $V = V_1 \cup \ldots \cup V_k$, where $V_1, \ldots, V_k$ are strongly disjoint varieties in $\mb{B}_d$. Then $f \in \cM_V$ if and only if for all $i=1, \ldots, k$, the restriction $f\big|_{V_i}$ is in $\cM_{V_i}$. 
\end{theorem} 
\begin{proof}
Since $\cM_W = \cM_d \big|_W$ for every variety $W$, the ``only if" implication is obvious. Assume therefore that for all $i=1, \ldots, k$, $f\big|_{V_i} \in \cM_{V_i}$. For all $i$, let $f_i \in \cM_d$ satisfy $f_i\big|_{V_i} = f\big|_{V_i}$. Let $\phi_1, \ldots, \phi_k \in \cA_d$ be as in the definition of strong disjointness . Then $F = f_1 \phi_1 + \ldots f_k \phi_k \in \cM_d$, and the restriction of $F$ to $V$ equals $f$. Thus $f \in \cM_V$. 
\end{proof}

\begin{theorem}
Let $V = V_1 \cup \ldots \cup V_k$ and $W = W_1 \cup \ldots \cup W_k$ be two varieties such that the $V_1, \ldots, V_k$ and $W_1, \ldots, W_k$ are strongly disjoint. If $\cM_{V_i}$ is isomorphic to $\cM_{W_i}$ for all $i=1, \ldots, k$, then $\cM_V$ is isomorphic to $\cM_W$. 
\end{theorem} 
\begin{proof}
By Theorem \ref{thmb1}, for all $i=1, \ldots, k$ there is a biholomorphism $H_i : W_i \rightarrow V_i$ with $\cM_d$ components, such that the isomorphism $\cM_{V_i} \cong \cM_{W_i}$ is implemented by $f \mapsto f \circ H_i$. Let $h$ be defined on $W$ by $h\big|_{W_i} = H_i$. Since for all $f \in \cM_V$ we have that $f \circ h \big|_{W_i} = f\big|_{V_i} \circ H_i$, we conclude from Theorem \ref{thm:restrictions} that $f \mapsto f \circ h$ implements an isomorphism of $\cM_V$ onto $\cM_W$. 
\end{proof}

\begin{remark}
It now follows from Theorem \ref{thmb1} that the function $h$ defined in the proof above can be extended to a function with components in $\cM_d$. 
\end{remark}

\bibliography{references}

\end{document}